\documentclass[12pt]{amsart}
\usepackage{amssymb,amscd}
\usepackage{verbatim}

\usepackage{amsmath,amssymb,graphicx,mathrsfs}   
\usepackage[colorlinks=true,allcolors = blue]{hyperref} 

\usepackage{tikz}
\usetikzlibrary{matrix}

\usepackage[all]{xy}

\let\frak\mathfrak

\def\>{\relax\ifmmode\mskip.666667\thinmuskip\relax\else\kern.111111em\fi}
\def\<{\relax\ifmmode\mskip-.333333\thinmuskip\relax\else\kern-.0555556em\fi}
\def\vsk#1>{\vskip#1\baselineskip}
\def\vv#1>{\vadjust{\vsk#1>}\ignorespaces}
\def\vvn#1>{\vadjust{\nobreak\vsk#1>\nobreak}\ignorespaces}

  \let\ssize\scriptstyle
\let\sssize\scriptscriptstyle

\let\Medskip\medskip
\def\medskip{\par\Medskip}
\let\Bigskip\bigskip
\def\bigskip{\par\Bigskip}

\let\Maketitle\maketitle
\def\maketitle{\Maketitle\thispagestyle{empty}\let\maketitle\empty}

\newtheorem{thm}{Theorem}[section]
\newtheorem{cor}[thm]{Corollary}
\newtheorem{lem}[thm]{Lemma}

\theoremstyle{definition}                                  
\numberwithin{equation}{section}

\theoremstyle{definition}
\newtheorem*{rem}{Remark}

\let\mc\mathcal
\let\nc\newcommand

\let\la\lambda

\let\phi\varphi

\let\Om\Omega

\let\der\partial

\let\ox\otimes

\let\geq\geqslant

\let\leq\leqslant

\let\on\operatorname
\let\bi\bibitem
\let\bs\boldsymbol

\def\C{{\mathbb C}}
\def\Z{{\mathbb Z}}

\def\F{{\mathbb F}}   

\def\+#1{^{\{#1\}}}

\def\tr{\on{tr}}
\def\Wr{\on{Wr}}

\def\beq{\begin{equation}}
\def\eeq{\end{equation}}
\def\be{\begin{equation*}}
\def\ee{\end{equation*}}

\nc{\bea}{\begin{eqnarray*}}
\nc{\eea}{\end{eqnarray*}}
\nc{\bean}{\begin{eqnarray}}
\nc{\eean}{\end{eqnarray}}
\nc{\Ref}[1]{{\rm(\ref{#1})}}

\let\ga\gamma

\nc{\Il}{{\mc I_{\bs\la}}}
\nc{\bla}{{\bs\la}}
\nc{\Fla}{\F_\bla}
\nc{\tfl}{{T^*\Fla}}
\nc{\GL}{{GL_n(\C)}}
\nc{\GLC}{{GL_n(\C)\times\C^*}}

\let\sd s 

\def\ddk_#1{\kk_{#1}\<\>\frac\der{\der\<\>\kk_{#1}}}

\def\FFF{\mathbb{F}}

\def\bul{\mathbin{\raise.2ex\hbox{$\sssize\bullet$}}}
\def\intt{\mathchoice
{\mathop{\raise.2ex\rlap{$\,\,\ssize\backslash$}{\intop}}\nolimits}
{\mathop{\raise.3ex\rlap{$\,\sssize\backslash$}{\intop}}\nolimits}
{\mathop{\raise.1ex\rlap{$\sssize\>\backslash$}{\intop}}\nolimits}
{\mathop{\rlap{$\sssize\<\>\backslash$}{\intop}}\nolimits}}

\let\kk q 
\let\cc c

\let\Ko K

\def\GZ/{Gelfand-Zetlin}
\def\KZ/{{\slshape KZ\/}}
\def\qKZ/{{\slshape qKZ\/}}
\def\XXX/{{\slshape XXX\/}}

\nc{\A}{{\mc C}}

\def\sll{{\frak{sl}}}

\def\Fpz{{\F_p(z)}}
\def\Fz{{\F_p(z)}}
\def\Fzx{{\F_p(z)[x]}}

\def\Mz{{\mc M_\Fz}}
\nc{\hsl}{\widehat{{\frak{sl}_2}}}

\nc{\BC}{{ \mathbb C}}
\nc{\lra}{\longrightarrow}
\nc{\CO}{{\mathcal{O}}}
\nc{\BZ}{{ \mathbb Z}}
\nc{\hfn}{\hat{\frak{n}}}

\begin{document}

\hrule width0pt
\vsk->

\title[An invariant subbundle   of the KZ connection mod $p$]
{An invariant subbundle   of the KZ connection mod $p$ 
\\
and  reducibility of
$\hsl$ Verma modules mod $p$}

\author[Alexander Varchenko]
{ Alexander Varchenko$^{\star}$}

\maketitle

\begin{center}
{\it Department of Mathematics, University
of North Carolina at Chapel Hill\\ Chapel Hill, NC 27599-3250, USA\/}

\vsk.5>
{\it Faculty of Mathematics and Mechanics, Lomonosov Moscow State
University\\ Leninskiye Gory 1, 119991 Moscow GSP-1, Russia\/}

\end{center}

\vsk>
{\leftskip3pc \rightskip\leftskip \parindent0pt \Small
{\it Key words\/}:  KZ equations,  reduction to
characteristic $p$, arithmetic solutions

\vsk.6>
{\it 2010 Mathematics Subject Classification\/}: 13A35 (33C60, 32G20) 
\par}

{\let\thefootnote\relax
\footnotetext{\vsk-.8>\noindent
$^\star\<$
{\it E\>-mail}:
anv@email.unc.edu\,, supported in part by NSF grant DMS-1665239}}

\begin{abstract}

 We consider the KZ differential equations over $\mathbb C$
in the case, when its multidimensional hypergeometric solutions are 
one-dimensional integrals. We also consider
the same differential equations over a finite field $\F_p$.
We study  the space of polynomial solutions of these differential equations over $\F_p$, 
constructed in a previous work by V.\,Schechtman and the  author.
The module of these polynomial solutions defines an invariant subbundle  of the associated KZ connection modulo $p$.
We describe the algebraic equations for that subbundle and argue that the equations correspond to highest weight vectors
of the associated $\hsl$ Verma modules over the  field $\F_p$.

\end{abstract}

{\small\tableofcontents\par}

\setcounter{footnote}{0}
\renewcommand{\thefootnote}{\arabic{footnote}}

\section{Introduction}

The KZ connection is a flat connection on a trivial complex vector bundle, whose fiber is the tensor
 product  of  Lie algebra modules,
see \cite{KZ, EFK}. The connection has an important invariant subbundle of conformal blocks defined in conformal field theory.
 The algebraic equations 
for this subbundle are useful for applications. The algebraic equations were described in \cite{FSV1, FSV2, FSV3}.
In \cite{SV1, SV2} flat sections of the KZ connection were constructed in the form of multidimensional 
hypergeometric integrals. It was shown in \cite{FSV1, FSV2, FSV3} that these hypergeometric flat sections 
always are sections of the subbundle of conformal blocks.

Recently in \cite{SV4} the KZ connection was considered over a finite field $\F_p$ and 
polynomial flat sections of the KZ connection were constructed
as $p$-analogs of the hypergeometric flat sections of the KZ connection over the field $\C$, 
see also \cite{V3, V4, V5, SliV2}. 
We call these sections the {\it arithmetic flat sections} of the KZ connection over $\F_p$.

The problem is to 
describe the algebraic equations for the subbundle spanned by the arithmetic flat sections.
We consider the
 particular case of the KZ connection in
 which the hypergeometric flat sections over the field $\C$
are given by one-dimensional integrals. We identify the annihilator of the subbundle of arithmetic flat sections with 
a certain space of rational functions in one variable $x$ over the field $\F_p(z)$. 
That space is related to the De Rham complex 
of rational differential forms on the curve of genus zero with punctures. That De Rham complex 
over the field $\C$ was studied in \cite{SV3}  to relate it to the highest weight vectors in reduced Verma modules
over the complex affine Lie algebra $\hsl$.

In the end of the paper we argue that the annihilator of the module of arithmetic 
flat sections is related to highest weight vectors in the Verma modules of $\hsl$ considered over the field $\F_p$.

\vsk.2>
The paper is organized as follows. In Section  \ref{sec KZ} we describe our example  of the KZ differential equations,
 construct its solutions over $\C$ and $\F_p$. In Section \ref{sec MAS}
 we discuss the annihilator of the module of arithmetic flat sections and identify it with a certain space of rational functions in 
one variable $x$. In Section \ref{sec sp el} we describe a set of generators in the annihilator, which gives us
a set of explicit algebraic equations satisfied by the arithmetic flat sections.
In Section \ref{sec ar and rep} we argue that the algebraic  equations for arithmetic  flat sections are related to 
highest weight vectors in the Verma modules of the affine Lie algebra $\hsl$ considered over $\F_p$.

\vsk.2>
The author thanks Vadim Schechtman and Alexey Slinkin for useful discussions.

\section{KZ equations}
\label{sec KZ}

\subsection{Description of equations} 
\label{sec DE} 

In this paper the numbers $p$, $q$ {\it are prime numbers,
$n$ a positive integer, $p>n$,  $n=kq+1$ for some} $k\in\Z_{>0}$.
 We  study the following system of equations
for a  column vector  $I(z)=(I_1(z)$, \dots, $I_{n}(z))$\,:
\bean
\label{KZ}
\phantom{aaa}
 \frac{\partial I}{\partial z_i} \ = \
   {\frac 1q} \sum_{j \ne i}
   \frac{\Omega_{ij}}{z_i - z_j}  I ,
\quad i = 1, \dots , n,
\qquad
I_1(z)+\dots+I_{n}(z)=0,
\eean
where $z=(z_1,\dots, z_{n})$,  
\bea
 \Omega_{ij} \ = \ 
\begin{pmatrix}
             & \vdots^i &  & \vdots^j &  
\\
        {\scriptstyle i} \cdots & {-1} & \cdots &            1   & \cdots 
\\
                   & \vdots &  & \vdots &  
\\
        {\scriptstyle j} \cdots & 1 & \cdots & {- 1}&
                 \cdots 
\\
                   & \vdots &  & \vdots &
                   \end{pmatrix} ,
\eea
and all other entries of $\Om_{ij}$ are zero. This
joint system of {\it differential and 
algebraic equations} is called the {\it system of KZ  equations} in this paper.

\vsk.2>

System  \Ref{KZ} is the system of the  classical KZ differential equations 
associated with the Lie algebra $\sll_2$ and the subspace of highest weight vectors of weight $n-2$ of the tensor 
product
$V^{\ox n}$\,,
where $V$ is the irreducible two-dimensional  $\sll_2$-module, up to a gauge transformation, see 
this example in  \cite[Section 1.1]{V2}. See also \cite{V1, SliV2}.

\vsk.2>
We consider system \Ref{KZ} over the field $\C$ and the field $\F_p$ with $p$ elements.

\subsection{Solutions  over $\C$} 
All solutions of \Ref{KZ} have the form
\bean
\label{hg sol}
I^{(\ga)}(z) =\left( \int_\ga \frac{\Phi_\C(x,z)}{x-z_1}\,dx,\ \dots\ , \int_\ga \frac{\Phi_\C(x,z)}{x-z_n}\,dx
\right),
\eean
where
\bean
\label{mast f}
\Phi_\C(x,z) = \prod_{a=1}^{n}(x-z_a)^{-1/q}
\eean
and  $\ga$
is an element  of the first homology group
 of the algebraic curve with affine equation
\bea
y^q = (x-z_1)\dots (x-z_{n})\,.
\eea
Starting from such $\ga$,  chosen for given $\{z_1,\dots,z_{n}\} \subset \C$,
 the vector $I^{(\ga)}(z)$ can be analytically continued,
 as a  multivalued holomorphic function of $z$, 
to the complement in $\C^n$ to the union of the
diagonal  hyperplanes $z_i=z_j$. 

The complex vector space of such integral solutions is the  $n-1$-dimensional vector space of all solutions of
system \Ref{KZ}. 
See these statements in the example in   \cite[Section 1.1]{V2}, also in  \cite{SliV2}.

\subsection{Solutions  over $\F_p$} Polynomial solutions of the general 
KZ differential equations over $\F_p$ were constructed in \cite{SV4}. 
The particular case of  system \Ref{KZ} was studied in \cite{SliV2}.
For $q=2$ system \Ref{KZ} was studied in \cite{V5}.

\vsk.2>
Denote by $a$ the unique integer such that 
\bean
\label{a_1}
1\leq a < q\quad \on{and}\quad  q\,\vert \,(ap-1)\,.
\eean
Denote
\bean
\label{M_i}
M = (ap-1)/q\,\in \Z_{>0}\,.
\eean
Let
\bean
\label{mp red}
\Phi(x,z) := \prod_{i=1}^n(x-z_i)^{M},
\eean
\bean
\label{P}
P(x,z) \,:=\,\Big(\frac {\Phi(x,z)}{x-z_1}, \dots,\frac {\Phi(x,z)}{x-z_n}\Big)\,=\, \sum_i P^{i}(z) \,x^i \,,
\eean
where $P^{i}(z)$ are $n$-vectors of polynomials in $z_1,\dots,z_n$ with coefficients in $\FFF_p$.

\begin{thm}[{\cite[Theorem 1.2]{SV4}}] \label{thm Fp} 
For any positive integer $l$, the vector of polynomials
\\
 $P^{lp-1}(z)$ 
is a solution of system \Ref{KZ}.

\end{thm}

Theorem \ref{thm Fp} is a particular case of \cite[Theorem 2.4]{SV4}. Cf. Theorem \ref{thm Fp} 
in \cite{K}. See also \cite{V3, V4, V5, SliV2}. 
\vsk.2>

The solutions  $P^{lp-1}(z)$ 
given by this construction are called  the {\it arithmetic solutions} of system \Ref{KZ}. They are $p$-analogs of the hypergeometric solutions
\Ref{hg sol}, in which the integration over a cycle $\ga$ in the integrals of formula \Ref{hg sol}
 is replaced by taking the coefficient of $x^{lp-1}$ in the Taylor expansion of the integrands.

\vsk.2>
Denote $\F_p[z^p]:=\F_p[z_1^p,\dots,z_{n}^p]$. 
The set of all polynomial solutions of system \Ref{KZ} with coefficients in $\F_p$
is a module over the ring
$\F_p[z^p]$ since equations \Ref{KZ} are linear and
$\frac{\der z_i^p}{\der z_j} =0$ in $\F_p[z]$ for all $i,j$.
\vsk.2>

The $\F_p[z^p]$-module
\bean
\label{Def M_M}
\mathcal{M}\,=\,\Big\{ \sum_{l} c_l(z) P^{lp-1}(z) \ |\ c_l(z)\in\F_p[z^p]\Big\},
\eean
spanned by arithmetic solutions,
 is called the {\it module of arithmetic solutions}.
\vsk.2>

The range for the index $l$ is defined by  the inequalities 
 $0< lp-1\leq nM-1$. This implies that  $l=1,\dots, ak$, see \cite[Lemma 5.1]{SliV2}.

\begin{thm}[{\cite[Theorem 3.2]{SliV2}}]
\label{thm li in}
The $\F_p[z^p]$-module $\mc M$ of arithmetic solutions is free of rank $ak$,
that  is, if \ $\sum_{l=1}^{ak} c_l(z) P^{lp-1}(z) = 0$ for some 
$c_l (z)  \in \mathbb{F}_p [z^p]$, then $c_l (z) = 0$ for all $l$.

\end{thm}

\section{Module of arithmetic solutions}
\label{sec MAS}

\subsection{Auxiliary fields and rings} 

For variables $u=(u_1,\dots,u_r)$, denote  by $\F_p(u) = 
\\
\F_p(u_1,\dots,u_r)$
the field of rational functions in $u_1,\dots,u_r$ with coefficients in $\F_p$.
Denote by $\F_p(u)[x]$ the ring of polynomials in $x$ with coefficients in $\F_p(u)$.
Denote by $\F_p(u^p)\subset \F_p(u)$ the subfield of rational functions in 
$u_1^p,\dots,u_r^p$. 

\subsection{Auxiliary lemma}

\begin{lem}
\label{lem der=0} Let $f(t)\in\F_p(t)$ be a rational function in one variable $t$ such that
\\
$f\rq{}(t):=\frac {d f}{dt}(t)=0$, then $f(t) \in \F_p(t^p)$\,.

\end{lem}

\begin{proof}
For $g,h\in \F_p(t)$ denote
\bea
\Wr(g,h) = gh\rq{}-g\rq{}h\,,
\eea
the Wronskian of $g,h$. If $a,b,c,d\in \F_p(t^p)$, then
\bean
\label{Wr lin}
\Wr (ag+bh, cg+dh) \,=\, (ad-bc)\, \Wr(g,h)\,.
\eean

If $f=g/h$ with $g, h\in \F_p[t]$, then $f\rq{} = \Wr(g,h)/g^2 $ and 
$\Wr(g,h) = 0$.

Using \Ref{Wr lin}  we show that there exists $j(t) \in \F_p[t]$ and $a,b\in \F_p(t^p)$ such that
$g=aj$, $h=bj$ and hence $g/h=a/b \in \F_p(t^p)$. 

Indeed, let $g(t) = g_l t^l + \dots + g_0$, $g_i\in \F_p$, $g_l\ne 0$, and
$h(t) = h_m t^m + \dots + h_0$, $h_i\in \F_p$, $h_m\ne 0$. Equation $\Wr(g,h)=0$ implies that $p\,|\,(l-m)$.

If $l\geq m$, consider the polynomial $\tilde g(t) = g(t) - t^{l-m}\frac{g_l}{h_m} h(t)$.  Then 
$\deg \tilde g(t) < \deg g(t)$ and 
$\Wr(\tilde g, h) = \Wr(g,h)$. 

If $m\geq l$, consider the polynomial $\tilde h(t) = h(t) - t^{m-l}\frac{h_m}{g_l} g(t)$.  Then 
$\deg \tilde h(t) < \deg h(t)$ and $\Wr(g, \tilde h) = \Wr(g,h)$. 

We  iterate this procedure, which decreases the degree 
of one of the two polynomials $h, g$  by a multiple of $p$ while keeping the Wronskian of the two polynomials
equal to zero.  The procedure will stop when one of the polynomials becomes zero. That means that the last nonzero polynomial $j(t) \in \F_p[t]$
divides the initial polynomials $g(t), h(t)$ and the ratios $a(t):=g(t)/j(t)$ and $b(t):=h(t)/j(t)$ belong to $\F_p[t^p]$. 
\end{proof}

\begin{cor} 

Let $f(u_1,\dots,u_r) \in \F_p(u_1,\dots,u_r)$ be
 such that $\frac{\der f}{\der u_i} =0$, $i=1,\dots,r$. Then $f(u_1,\dots,u_r)\in\F_p(u_1^p,\dots,u_r^p)$.
\qed
\end{cor}

\subsection{Vector subspace $\Mz$}

Let
\bea
\mc M_\Fz\,: =\,\mc M\ox_{\F_p[z^p]} \Fpz \ \subset \  \Fpz^n\,,
\eea
be the $\Fz$-vector subspace of $\Fpz^n$ spanned by arithmetic solutions.

\vsk.2>

Recall that the {\it KZ connection} is defined by the commuting differential operators
$\nabla_i$, $i=1,\dots,n$,
\bean
\label{nubla}
\nabla_i :=  \frac{\partial }{\partial z_i} -
   {\frac 1q} \sum_{j \ne i}
   \frac{\Omega_{ij}}{z_i - z_j} \ :\  \Fpz^n\ \to\  \Fpz^n\,.
\eean
The arithmetic solutions $P^{lp-1}(z)$ are flat sections of the KZ connection.

\begin{lem}
\label{lem KZ inv}
The subspace $\Mz$ is invariant with respect to the KZ connection, namely,
for any element $I = \sum_{l=1}^{ak} c_l(z) P^{lp-1}(z)$, $c_l(z) \in \Fz$,  of $\Mz$ we have
\bea
\nabla_i I \ =\  \sum_{l=1}^{ak} \,\frac {\der c_l}{\der z_i}(z)\, P^{lp-1}(z)\,\in \, \Mz\,,
\qquad i=1,\dots,n\,.
\eea
\qed
\end{lem}

\begin{lem}
Let $I^1(z),\dots, I^l(z)$ be flat sections of the KZ connection linearly independent over the field
$\F_p(z^p)$. Then the $\Fz$-vector subspace of $\Fz^n$ spanned by $I^1(z),\dots, I^l(z)$ is of dimension $l$.

\end{lem}

\begin{proof}
The proof is by induction. The statement is true for $l=1$. Assume that it is true for all $l< m$ and prove it for $l=m$.
Assume that  $I^1(z),\dots, I^m(z)$ are linearly independent over $\F_p(z^p)$ but  dependent over $\F_p(z)$:
\bean
\label{ldp}
c_1(z) I^1(z) + \dots + c_m(z) I^m(z) =0
\eean
 for some $c_j(z) \in \F_p(z)$ not all equal to zero. 
If at least one of the coefficients $c_j(z)$ is zero, then \Ref{ldp} contradicts to
the induction assumption. Hence all $c_j(z)$ are 
nonzero. Divide \Ref{ldp} by $c_1(z)$ and obtain the new relation of the form
\bean
\label{ldp2}
I^1(z) + c_2(z) I^2(z)+\dots + c_m(z) I^m(z) =0 \,.
\eean
Apply $\nabla_i$ to \Ref{ldp2} and obtain
\bean
\label{ldp3}
\frac{\der c_2}{\der z_i} I^2(z)+\dots + \frac{\der c_n}{\der z_i} I^m(z) =0 \,.
\eean
If  $\frac{\der c_j}{\der z_i}\ne 0$ for some $i$ and $j$, then equation \Ref{ldp3} 
for that $i$ contradicts to the induction assumption.
If $\frac{\der c_j}{\der z_i}= 0$ for all $i,j$, then $c_j(z) \in \F_p(z^p)$ 
for all $j$ by Lemma \ref{lem der=0}. Hence \Ref{ldp} gives a linear dependence of
$I^1(z),\dots, I^m(z)$  over $\F_p(z^p)$, which contradicts to the assumptions of 
the lemma.
\end{proof}

\begin{cor}
\label{cor dim}
We have $\dim_{\Fz} \,\Mz \ =\  \on{rank}_{\F_p[z^p]} \mc M\,=\, ak$\,.

\end{cor}

\subsection{Annihilator of $\Mz$}
Consider the $\Fz$-vector space $(\Fz^n)^*$ dual to $\Fz^n$ and the subspace
\bea
\on{Ann}(\Mz) =\{ c\in (\Fz^n)^*\ |\  \langle c, I\rangle=0\ \forall I\in\Mz\}\,.
\eea
This subspace is  identified with the subspace
\bean
\label{cP}
\Big\{(c_1(z),\dots, c_n(z))\in\Fz^n\ \Big|\
\sum_{j=1}^n c_j(z) P^{lp-1}_j(z) = 0\,,\ \  l=1,\dots,ak\Big\},
\eean
where
$P^{lp-1}(z) = (P^{lp-1}_1(z),\dots, P^{lp-1}_n(z))$, $l=1,\dots, ak$, are arithmetic solutions.

\vsk.2>
We have
\bean
\label{dim Ann}
\dim_\Fz \on{Ann}(\Mz)  = n-ak = (q-a)k  +1\,,
\eean
by Corollary \ref{cor dim}.

\begin{thm}
\label{thm 1}
We have $(c_1(z), \dots,c_n(z))\in\on{Ann}(\Mz)$
if and only if
 there exists a polynomial
$Q(x,z) \in \F_p(z)[x]$ such that 
\bean
\label{derr}
\frac{\der Q}{\der x}(x,z)\, =\, \sum_{j=1}^n\,c_j(z)\, \frac {\Phi(x,z)}{x-z_j}\,.
\eean
\end{thm}

\begin{proof} 

The polynomial $Q(x,z)$ with property \Ref{derr} exists if and only if 
the Taylor expansion of the polynomial $\sum_{j=1}^n c_j(z) \frac {\Phi(x,z)}{x-z_j}\,$ 
with respect to $x$ has zero coefficients
for all of the monomials $x^{lp-1}$, $l=1, \dots, ak$. That property is equivalent to the property that 
$(c_1(z), \dots,c_n(z))\in\on{Ann}(\Mz)$\,.
\end{proof}

For $m\in\Z_{\geq 0}$ denote by $\Fzx_{\leq m}$ the $\Fz$-vector space of polynomials in 
$\Fzx$ of degree $\leq m$ in $x$.

\vsk.2> Clearly, $\deg_x\frac {\Phi(x,z)}{x-z_j} = Mn-1$, where $M$ is defined in \Ref{M_i}. 
Hence a polynomial $Q(x,z)$ in Theorem \ref{thm 1}, if exists,
can be chosen to be in $\Fzx_{\leq Mn}$.

Define the map
\bean
\label{der}
\der
\,:\, \Fzx_{\leq Mn}\,\to\,\Fzx_{\leq Mn-1}\,,\qquad Q(x,z) \,\mapsto \frac{\der Q}{\der x}(x,z)\,.
\eean
Define the $n$-dimensional vector space
\bea
\Om_{\log} =\Big\{ \sum_{j=1}^n c_j(z) \frac{\Phi(x,z)}{x-z_j}\ \Big|\ (c_1(z),\dots,c_n(z))\in\Fz^n\Big\},
\eea
the span of the polynomials $\frac{\Phi(x,z)}{x-z_j}$, $j=1,\dots,n$. 

\begin{cor}
\label{cor Ader}
 Theorem \Ref{thm 1} gives an isomorphism
\bean
\label{iso Ader}
\on{Ann}(\Mz)\  \simeq \ \Om_{\log} \cap \der
 \big( \Fzx_{\leq Mn}\big)\,.
\eean
\qed
\end{cor}

\begin{cor}
\label{cor dim cap}

We have
\bean
\label{der dim}
\dim_\Fz\,\Big(\Om_{\log} \cap \der ( \Fzx_{\leq Mn})\Big)\ = \ (q-a)k+1\,.
\eean

\end{cor}

Corollary \ref{cor dim cap} follows from Corollaries \ref{cor Ader}, \ref{cor dim}, and Theorem \ref{thm li in}.

\begin{rem}
The map $\der :  \Fzx_{\leq Mn} \to \Fzx_{\leq Mn-1}$ \,
 in \Ref{der} is a part of the larger De Rham complex
\bean
\label{dr com}
 \Omega \overset{\der}{\lra}\Omega \,,
\eean
where $\Om$ is the $\Fz$-vector space with basis
\bean
\label{def OM}
\phantom{aaa}
\frac{\Phi(x,z)}{(x-z_i)^{m}}\ \  \on{for}\  m\in\BZ_{>0},\ i=1,\dots,n, \qquad\on{and}\qquad
{\Phi(x,z)} x^m\ \ \on{for}\  m\in\BZ_{\geq 0} \,,
\eean 
and $\der = \frac{\der}{\der x}$.

\end{rem}

\section{Spanning elements of $\Om_{\log} \cap \der \big( \Fzx_{\leq Mn}\big)$}
\label{sec sp el}

\begin{lem}
\label{lem bases}
The polynomials
\bean
\label{set sp}
 \frac{\Phi(x,z)}{(x-z_j)^m}\,,\qquad j=1,\dots,n\,,\qquad 0< m\leq M\,,
\eean
form a basis of $\Fzx_{\leq Mn-1}$\,. The polynomials in \Ref{set sp} together with the polynomial
$\Phi(x,z)$ form a basis of $\Fzx_{\leq Mn}$\,.

\end{lem}

\begin{proof}
Each of the polynomials in \Ref{set sp} lies in $\Fzx_{\leq Mn-1}$. 
They are linearly independent and their number $Mn$ equals
$\dim \Fzx_{\leq Mn-1}$. Hence they form a basis of $\Fzx_{\leq Mn-1}$. Similarly for the second part of the lemma.
\end{proof}

In these bases the map $\der$ defined in \Ref{der} is given by the formulas:
\bean
\label{bbb}
\frac{\der}{\der x} \Big(\Phi(x,z)\Big) 
&=& M\,\sum_{j=1}^n \frac{\Phi(x,z)}{x-z_j}\,,
\\
\label{bff}
\phantom{aaaa}
\frac{\der}{\der x} \Big(\frac{\Phi(x,z)}{(x-z_i)^{m}}\Big) 
&=& (M-m)\,\frac{\Phi(x,z)}{(x-z_i)^{m+1}}
\,-\, M\sum_{l=1}^m\Big(\sum_{j\neq i}\, \frac{1}{(z_j-z_i)^l}\Big)\, 
\frac{\Phi(x,z)}{(x-z_i)^{m+1-l}}
\\
\notag
&&
+\,M\,\sum_{j\neq i}\,\frac{1}{(z_j-z_i)^m}  \,\frac{\Phi(x,z)}{x-z_j}\,,
\eean
for $i=1,\dots,n$\ and \ $m\geq 1$.  In particular for $m=M$ we have
\bean
\label{bffM}
\phantom{aaaa}
\frac{\der}{\der x} \Big(\frac{\Phi(x,z)}{(x-z_i)^{M}}\Big) 
&=& 
-\, M\sum_{l=1}^M\Big(\sum_{j\neq i}\, \frac{1}{(z_j-z_i)^l}\Big)\, 
\frac{\Phi(x,z)}{(x-z_i)^{M+1-l}}
\\
\notag
&&
+\,M\,\sum_{j\neq i}\,\frac{1}{(z_j-z_i)^M}  \,\frac{\Phi(x,z)}{x-z_j}\,.
\eean

\begin{lem}
\label{lem Qi}

For $i=1,\dots,n$, there exist unique $A_{i,1}(z),\dots, A_{i,M-1}(z) \in\F_p(z)$
such that the polynomial
\bean
\label{Qi}
Q_i(x,z)\, \,=  \frac {\Phi(x,z)}{(x-z_i)^M} \,+\,\sum_{m=1}^{M-1}\, A_{i,m}(z)\,\frac {\Phi(x,z)}{(x-z_i)^{m}}\, 
\eean
in $x$ has the property: $\frac{\der Q_i}{\der x} \in \Om_{\log}$, that is,
\bean
\label{property}
\frac{\der Q_i}{\der x}(x,z)\,=\,\sum_{j=1}^n\,B_{i,j}(z)\,\frac{\Phi(x,z)}{x-z_j}\,,
\eean
for suitable $B_{i,j}(z) \in \F_p(z)$.

\end{lem}

\begin{proof}
Let $Q_i(x,z)$ be an expression as in  \Ref{Qi} for some $(A_{i,m}(z))_{m=1}^{M-1}$.  Applying formula \Ref{bff} 
we obtain
\bea
\frac{\der Q_i}{\der x}(x,z) \,
= \, \sum_{m=1}^M\,C_m(z)\, \frac{\Phi(x,z)}{(x-z_i)^m} \,+\,
 \sum_{j\ne i}\,B_{ij}(z)\,\frac{\Phi(x,z)}{x-z_j}\,,
\eea 
where $C_m(z), B_{ij}(z)$ are linear expressions in $A_{i,1}(z),\dots,A_{i,M-1}(z)$ with coefficients in $ \F_p(z)$. We have
\bean
\label{CM}
C_M(z) 
&=& 
A_{i,M-1}(z)\,-\, \sum_{j\neq i}\, \frac{M}{z_i-z_j} \,,
\eean
and for $m=2,\dots, M-1$,
\bean
\label{Ca}
C_m(z)
& =& (M-m+1) A_{i,m-1}(z) \ + \ \dots\,\ ,
\eean
where the dots denote a linear expression in  $A_{i,m}(z), \dots,A_{i,M-1}(z)$.
To obtain property \Ref{property} it is necessary and sufficient to choose
the coefficients  $A_{i,1}(z), \dots,A_{i,M-1}(z)$
so that $C_m(z) = 0$ for $m=2,\dots, M$. But this can be done uniquely by formulas
\Ref{CM} and \Ref{Ca}.
\end{proof}

Denote
\bean
\label{CD}
C_{i,m}(z) = \sum_{j=1, \atop j\neq i}^n\, \frac{M}{(z_j-z_i)^{m}}\,,
\qquad
D_{i,m}(x,z) = \sum_{j=1, \atop j\neq i}^n\, \frac{M}{(z_j-z_i)^{m}}\, \frac{\Phi(x,z)}{x-z_j}\,.
\eean

\begin{thm}
\label{thm SV}
For $i=1,\dots, n$, let $Q_i(x,z)$ be the polynomial determined  in Lemma \ref{lem Qi}.
Then
\bean
\label{Qii}
Q_i(x,z)
&=&
 \sum_{m\geq 0}\, 
\sum_{l_0+\ldots +l_m=M,\atop
l_0,\ldots, l_m>0}
\frac{\Phi(x,z)}{(x-z_i)^{l_0}} 
\,\prod_{r=1}^m \,\frac{C_{i, l_r}(z)}{l_1+\dots+l_r} \,.
\\
\label{Qi rel}
\frac{\der Q_i}{\der x}(x,z) 
& =& 
\sum_{m\geq 0}\, 
\sum_{l_0+\ldots +l_m=M,\atop
l_0,\ldots, l_m>0}
 \left(D_{i, l_0}(x,z) - C_{i,l_0}(z) \frac{\Phi(x,z)}{x-z_i} \right)
\,\prod_{r=1}^m \,\frac{C_{i, l_r}(z)}{l_1+\dots+l_r} \,.
\eean

\end{thm}

This theorem is a modification of \cite[Corollary 6.4]{SV3},
where formula \Ref{Qi rel} is considered over $\C$.
Notice that formula (40) in \cite{SV3} has misprints,
cf. formulas (40) and \Ref{Qi rel}.

\begin{proof} Using formula \Ref{bff} we eliminate from formula \Ref{bffM} all
the terms  $\frac{\Phi(x,z)}{(x-z_i)^l}$ with $l>1$. This leads  to formulas \Ref{Qii} and  \Ref{Qi rel}.

For example,
 if $q=2$, $p=5$,  then  $M=(p-1)/2=2$. Formulas \Ref{bffM} and \Ref{bff} take the form
\bea
\frac{\der}{\der x} \Big(\frac{\Phi}{(x-z_i)^{2}}\Big) 
&=&
 - C_{i,1} \frac{\Phi}{(x-z_i)^{2}} -C_{i,2} \frac{\Phi}{x-z_i} + D_{i,2}\,,
\\
\frac{\der}{\der x} \Big(\frac{\Phi}{x-z_i}\Big)
&=&
 \frac{\Phi}{(x-z_i)^{2}} - C_{i,1}\frac{\Phi}{x-z_i} + D_{i,1}\,,
\eea
and imply the formula
\bea
\frac{\der}{\der x}\Big(\frac{\Phi}{(x-z_i)^{2}} + C_{i,1} \frac{\Phi}{x-z_i}\Big) =
 \Big(D_{i,2} - C_{i,2} \frac{\Phi}{x-z_i}\Big) + \Big(D_{i,1} - C_{i,1} \frac{\Phi}{x-z_i}\Big) C_{i,1}\,.
\eea
\end{proof}

\begin{thm}
\label{thm sl2}
Assume that $F(x,z)\in \F_p(z)[x]_{\leq Mn}$   is such that
\bean
\label{dF}
\frac{\der F}{\der x}(x,z)\,=\, \sum_{j=1}^n\,A_j(z)\,\frac{\Phi(x,z)}{x-z_j}
\eean
for some $A_j(z) \in\F_p(z)$. Then there exist unique $D_0(z), D_1(z), \dots,D_n(z) \in \F_p(z)$ such that
\bean
\label{PDder}
F(x,z)\,=\,D_0(z)\Phi(x,z) \,+\,\sum_{j=1}^n D_j(z) Q_j(x,z)\,.
\eean
\end{thm}

\begin{proof}
We have
\bea
F(x,z) = c_0(z) \Phi(x,z) + \sum_{j=1}^n\sum_{m=1}^M c_{j,m}(z)\frac{\Phi(x,z)}{(x-z_j)^m}\,
\eea
for suitable coefficients $c_0(z), c_{j,m}(z)$ in  $\Fz$. 
Then
\bea
F(x,z) - \sum_{j=1}^n c_{j,M}(z)\,Q_j(x,z) =  
c_0(z) \,\Phi(x,z) + \sum_{j=1}^n\sum_{m=1}^{M-1} \tilde c_{j,m}(z)\frac{\Phi(x,z)}{(x-z_j)^m}\,
\eea
and
\bea
\frac{\der}{\der x}\Big(F(x,z) - \sum_{j=1}^n c_{j,M}(z)\,Q_j(x,z)\Big)\, =\, 
\sum_{j=1}^n\,\tilde A_j(z)\,\frac{\Phi(x,z)}{x-z_j}\,,
\eea
for suitable $\tilde A_j(z)$. 
Then formula \Ref{bff} implies that
\bea
F(x,z)\,-\, \sum_{j=1}^n c_{j,M}\,Q_j(x,z)\, =\, c_0(z)\, \Phi(x,z)\,.
\eea
The theorem is proved.
\end{proof}

\subsection{Formula \Ref{der dim}}
Here is a proof of formula \Ref{der dim} independent of Theorem \ref{thm li in}.

Denote by $\mc Q$ the $n+1$-dimensional $\Fz$-vector space  of linear combinations of the polynomials
$\Phi(x,z)$ and $Q_i(x,z)$, $i=1,\dots,n$. 

By Theorem \ref{thm sl2} 
\bea
\mc Q\,=\,\der^{-1}\Big(\Om_{\log} \cap \der \big( \Fzx_{\leq Mn}\big)\Big) \,.
\eea
 The kernel of the map
\bean
\label{mc Q}
\der|_{\mc Q}\,:\, \mc Q \,\to\, \Om_{\log}
\eean
is the subspace generated by monomials $x^{lp}$  with  $0\leq lp \leq Mn$.
We have
\bea
0\leq l\leq \frac{Mn}p= \frac{ap-1}{pq} (kq+1)= ak +\frac aq - \frac{kq+1}{pq}\,.
\eea
Using the fact that $p>q$, $p> n=kq+1$, $q>a>0$, we conclude that $l=0, 1,\dots, ak$. 
Hence the kernel of the map in \Ref{mc Q} is $ak+1$-dimensional and 
its image $\Om_{\log} \cap \der \big( \Fzx_{\leq Mn}\big)$ is of dimension
$n+1-(ak+1) = (q-a)k+1$. This statement is the statement of formula \Ref{der dim}.

It is easy to see that this independent proof of formula \Ref{der dim} implies  Theorem \ref{thm li in}.

\subsection{Corollary of Theorems \ref{thm 1}, \ref{thm SV},  \ref{thm sl2}}

\begin{cor}
\label{cor SV}
 For any arithmetic solution $P^{mp-1}_{}(z) = (P^{mp-1}_{1}(z),\dots, P^{mp-1}_{n}(z))$,
$m=1,\dots, ak$, we have the relation
\bean
\label{der Phi}
P^{mp-1}_{1}(z)\, +\dots +\, P^{mp-1}_{n}(z) \,=\,0,
\eean
coming from formula \Ref{bbb}
and for any $i=1,\dots,n$ 
the relation
\bean
\label{Qi Rel}
&&
\\
\notag
\sum_{m\geq 0}\, 
\sum_{l_0+\ldots +l_m=M,\atop
l_0,\ldots, l_m>0}\!\!
 \Big(\sum_{j=1, \atop j\neq i}^n\, \frac{M}{(z_j-z_i)^{l_0}}\, P^{mp-1}_j(z) - C_{i,l_0}(z)  P^{mp-1}_i(z) \Big)
\!\prod_{r=1}^m \,\frac{C_{i, l_r}(z)}{l_1+\dots+l_r} \,=\,0\,,
\eean
coming from formula \Ref{Qi rel}. Moreover, these relations generate the space $\on{Ann}(\Mz)$.
\qed

\end{cor}

\section{The space $\Om_{\log} \cap \der  \big( \Fzx_{\leq Mn}\big)$ and coinvariants mod $p$}
\label{sec ar and rep}

\subsection{Complex $(\Om, \der)$ over $\C$}

Let $\Om_\C$ be the $\C(z)$-vector space with basis
\bean
\label{def Om}
\phantom{aaa}
\frac{\Phi(x,z)}{(x-z_i)^{m}}\ \  \on{for}\  m\in\BZ_{>0},\ i=1,\dots,n, \qquad\on{and}\qquad
{\Phi(x,z)} x^m\ \ \on{for}\  m\in\BZ_{\geq 0} \,,
\eean 
where the polynomial $\Phi(x,z)$ is defined in \Ref{mp red}. 
Consider the new complex 
\bean
\label{dr coC}
 \Omega_\C \overset{\der}{\lra}\Omega_\C \,,
\eean
where $\der=\frac{\der}{\der x}$.  
Complexes \Ref{dr com} and \Ref{def Om} are given by the same formulas but are defined over different fields
$\Fz$ and $\C(z)$, respectively.

\subsection{Lie algebra $\hsl$} 
Let $\sll_2$ be the Lie algebra over $\C(z)$ 
of the  $2\times 2$-matrices with zero trace. Let  $e$,
$f$, $h$ be standard generators subject 
to the relations
$$
[e,f]=h,\qquad [h,e]=2e,\qquad  [h,f]=-2f.
$$
Let $\hsl$ be the  affine Lie algebra 
$\hsl=\sll_2[T,T^{-1}]\oplus\C(z) c$ with the bracket 
$$
[aT^i,bT^j]\ =\ [a,b]T^{i+j}\ +\ i\,\langle a,b\rangle\,\delta_{i+j,0}\, c ,
$$
where $c$ is central element, $\langle a,b\rangle=\tr (ab)$. 
Set 
\bea
e_1
&=& e, 
\qquad\ \
f_1  =f,\qquad \ \ \ \ h_1=h,
\\
e_2&=&fT,
 \qquad  
 f_2=eT^{-1},\quad \ h_2=c-h.
\eea
 These 
are the standard Chevalley generators defining $\hsl$ as the 
Kac-Moody algebra corresponding to the Cartan matrix
$
\left(\begin{array}{cc}
\phantom{a}2&-2
\\
-2&\phantom{a}2
\end{array}\right).$

The Lie algebra $\hsl$ has an 
automorphism $\pi$,
\bean
\label{aut pi}
\pi : \ \  c \mapsto c,
\quad eT^i \mapsto fT^i , 
\quad
fT^i\mapsto eT^i, \quad
hT^i\mapsto -hT^i.
\eean

\subsection{Verma modules} 

We fix $K\in\C$ and assume that the central element $c$ acts on all our representations 
by multiplication by $K$. 

For $L\in\C$, let $V(L,K-L)$ be the $\hsl$ {\em Verma module} with generating vector $v$.
It is the $\C(z)$-vector space  generated by $v$  subject to the relations 
\bea
e_1v=0,\qquad e_2v=0,\qquad  h_1v=L,\qquad  h_2v=(K-L)v\,.
\eea

The Verma module $V(L,K-L)$ is reducible if and only if at least one of the following  relations holds:
\bean
\label{r1}
&&
L-l+1+(m-1)(K+2) = 0\,,
\\
\label{r2}
&&
L+l+1-m(K+2)=0\,,
\\
\label{r3}
&&
K+2=0\,,
\eean
where $l,m\in\Z_{>0}$, see \cite{KK, MFF}.

Let $V(L,K-L)^*$ be the $\C(z)$-vector space dual to $V(L,K-L)$. 
The space $V(L,K-L)^*$ is an $\hsl$-module with the $\hsl$-action defined by the formulas:
\bea
\langle f_i\phi,x\rangle=\langle \phi,e_ix\rangle,
\qquad 
\langle e_i\phi,x\rangle=\langle\phi,f_ix\rangle ,
\eea
 where $\phi\in V(L,K-L)^*,\ x\in V(L,K-L),\ 
i=1,2 $.

\subsection{Lie algebra $\sll_2(z)$} 
\label{conf block}

Let $\sll_2(z)$ be the Lie algebra over $\C(z)$ of the
$\sll_2$-valued rational functions in $x$ of the form
$e\ox u_1 + h\ox u_2+ f\ox u_3$ with  $u_i\in\Om_\C$,
and the bracket is defined by the formula $[x\ox u_1, y\ox u_2] = [x,y]\ox (u_1u_2)$.

Let $W_1,\ldots,W_{n+1}$ be $\hsl$-modules. Then 
$\sll_2(z)$ acts on $W_1\otimes\ldots\otimes W_{n+1}$ by the formula
\bean
\label{action} 
&&
y\ox u(x) \cdot (w_1\otimes\ldots\otimes w_{n+1})
=  \big([y\ox u(x)]^{(z_1)} w_1\big)
\otimes w_2\otimes\ldots\otimes w_{n+1}+ \dots 
\\
\notag
&&
\phantom{aaaaaaaaa}
 +w_1\otimes\ldots\otimes w_{n-1}\otimes \big([y\ox u(x)]^{(z_n)} w_n\big)\otimes 
 w_{n+1}
\\
\notag
&&
\phantom{aaaaaaaaaaaa}
+ w_1\otimes\ldots\otimes w_n\otimes \big(\pi([y\ox u(x)])^{(\infty)} w_{n+1}\big) ,
\eean
where for $y\otimes u(x) \in\sll_2(z)$ the symbol $[y\otimes u(x)]^{(z_j)}$ 
denotes the Laurent expansion of $y\otimes u(x)$ with respect to $x$ at $x=z_j$ for $j=1,\dots,n,$
and $[y\otimes u(x)]^{(\infty)}$ 
denotes the Laurent expansion of $y\otimes u(x)$ with respect to $x$ at $x=\infty$; 
the symbol
$\pi$ in the last term denotes the $\hsl$-automorphism  defined in 
\Ref{aut pi}.

The $\hsl$-action is the map
\begin{equation}
\label{mult}
 \sll_2(z)\otimes \big(\otimes_{j=1}^{n+1} W_j\big)\ \overset{\mu}{\lra} \
\otimes_{j=1}^{n+1} W_j\,,
\end{equation}
denoted by $\mu$.
The quotient space $\big(\otimes_{j=1}^{n+1} W_j\big)\big/\on{Im}(\mu)$ is called the space of {\it coinvariants}
or {\it conformal blocks } at genus 0 with marked points $(x=z_j$, $W_j)$, $j=1,\dots,n$,
$(x=\infty$, $W_{n+1})$. See this construction, for example,  in \cite{FSV1, FSV2, FSV3, SV3, SliV1}.
One defines the KZ connection on the bundle of coinvariants with respect to changing $z_1,\dots,z_n$. Since the objects,
 considered in the previous sections
of this paper, are related to the KZ connection mod $p$ it is not surprising that they are related to the space 
of coinvariants.

\subsection{Complexes \Ref{dr coC} and \Ref{mult}}
Define
\bean
\label{kappa}
&&
\phantom{aaaaaaaaa}
K:=q-2\,,
\\
\label{WW}
&&
\ox_{j=1}^{n+1}W_j \,:=\, \big(V(1-ap, \,q+ap-3)^*\big)^{\ox n}\ox V(n-nap-2,\, q+nap-n)^*\,.
\eean
In \cite{SV3, SliV1}  the  commutative diagram
\begin{equation}
\label{diag}
  \xymatrix@R+0.3em@C+1em{
  \Om_\C \ar[r]^-\der \ar[d]_-{\nu^0} & \Om_\C \ar[d]^-{\nu^1}
 \\
  \sll_2(z)\otimes \big(\otimes_{j=1}^{n+1} W_j\big) \ar[r]_-\mu & \otimes_{j=1}^{n+1} W_j
  }
 \end{equation}
was defined, see there the definition of  $\nu^0, \nu^1$.

\vsk.2>
The Verma module $V(1-ap, q+ap-3)$, whose dual is
used in the definition \Ref{WW}, is reducible, due to equation \Ref{r1} for
\bean
\label{rc}
L=1-ap,\qquad l=1,\qquad
K =q-2,\qquad
m-1 = (ap-1)/q\,.
\eean
For $i=1,\dots,n$,\, let
\bean
\label{QiiC}
Q_{i}(x,z)
&=&
 \sum_{m\geq 0}\, 
\sum_{l_0+\ldots +l_m=M,\atop
l_0,\ldots, l_m>0}
\frac{\Phi(x,z)}{(x-z_i)^{l_0}} 
\,\prod_{r=1}^m \,\frac{C_{i, l_r}(z)}{l_1+\dots+l_r} \, \in \,\C(z)[x]\,.
\eean
This is the same polynomial as in formula \Ref{Qii} but considered over a different field.
This polynomial $Q_{i}(x,z)$ and its derivative   produce elements
\bea
\nu^0(Q_{i}(x,z))\in  \sll_2(z)\otimes \big(\otimes_{j=1}^{n+1} W_j\big)\,,
\qquad
\nu^1\Big(\frac{\der Q_{i}}{\der x}(x,z)\Big)\in  \otimes_{j=1}^{n+1} W_j\,.
\eea
It was explained in \cite{SV3, SliV1} that
 these elements can be constructed purely in terms of action of $\hsl(z)$ on $\otimes_{j=1}^{n+1} W_j$ by
using the fact that the module
$V(1-ap, q+ap-3)$,
 corresponding to the $i$-th factor of this tensor product, is reducible with the reducibility condition
\Ref{rc}. 

\vsk.2>
Summarizing these remarks, we may conclude that the linear relations \Ref{Qi Rel} for the arithmetic solutions of our KZ equations \Ref{KZ} correspond to  diagram \Ref{diag}, in which the affine Lie algebra $\hsl(z)$
is reduced modulo $p$ as well as its Verma modules. 

\begin{rem}
Notice that the Verma module $ V(n-nap-2,\, q+nap-n)$, whose dual is
used in the definition \Ref{WW}, is also reducible, due to equation \Ref{r2} for
\bean
\label{rc}
L=n-nap-2,\qquad l=1,\qquad
K =q-2,\qquad
m = Mn\,,
\eean
cf. \cite[Formula (41)]{SV3}.

\end{rem}

\end{document}